\documentclass[11pt,a4paper]{article}

\usepackage[british]{babel}
\usepackage[utf8]{inputenc}

\usepackage{tabularx}
\usepackage{array}
\usepackage{wrapfig}
\usepackage{amssymb,amsthm,thmtools}
\usepackage[tbtags]{amsmath}
\numberwithin{equation}{section}
\usepackage{relsize}
\usepackage{mathtools}	
\usepackage{accents}
\usepackage{multirow}
\usepackage{graphicx}
\usepackage{amscd}
\usepackage{epic, eepic}
\usepackage{url}
\usepackage{xcolor}
\usepackage{todonotes}
\usepackage{enumerate}
\usepackage{enumitem}
\usepackage{comment}
\usepackage{hhline}
\usepackage[hidelinks]{hyperref}
\addto\extrasbritish{%
}
\usepackage{tikz}
\usepackage{ifthen}
\usepackage{pgfplots}
\usepackage{mathrsfs}
\pgfplotsset{compat=1.15}
\usetikzlibrary{arrows}
\definecolor{ffffvc}{rgb}{1.,1.,0.3607843137254902}
\definecolor{wwzzff}{rgb}{0.4,0.6,1.}

\usepackage{tikz-cd}

\newcommand\PG{{\rm PG}}

\newcommand\cG{{\mathcal G}}

\newcommand\cL{{\mathcal L}}

\newcommand\cP{{\mathcal P}}

\newcommand\cQ{{\mathcal Q}}

\newcommand\PP{{\mathbb P}}
\newcommand\EE{{\mathbb E}}

\newcommand{\zz}{\mathbb{Z}}

\newcommand{\eps}{\varepsilon}
\renewcommand{\epsilon}{\eps}

\newcommand{\N} {\mathbb N}

\declaretheorem[style=plain,name=Theorem,numberwithin=section]{theorem}
\declaretheorem[style=plain,name=Proposition,sibling=theorem]{prop}
\declaretheorem[style=plain,name=Lemma,sibling=theorem]{lemma}
\declaretheorem[style=plain,name=Problem,sibling=theorem]{problem}
\declaretheorem[style=plain,name=Corollary,sibling=theorem]{cor}

\declaretheorem[style=definition,name=Definition,sibling=theorem]{defi}

\declaretheorem[style=definition,name=Remark,sibling=theorem]{remark}

\allowdisplaybreaks

\title{Settling the no-$(k+1)$-in-line problem when $k$ is not small}

\author{Benedek Kovács\thanks{ELTE Linear Hypergraphs Research Group, Eötvös Loránd University, Budapest, Hungary \& HUN-REN Alfréd Rényi Institute of Mathematics. Supported by the EKÖP-25 University Research Scholarship Program of the Ministry for Culture and Innovation from the source of the National Research, Development and Innovation Fund and the University Excellence Fund of Eötvös Loránd University.
		E-mail: {\tt kovacsbenedek@renyi.hu}}
	\and Zolt\'an L\'or\'ant Nagy\thanks{ELTE Linear Hypergraphs  Research Group,
		E\"otv\"os Lor\'and University, Budapest, Hungary. The author is supported by the University Excellence Fund of Eötvös Loránd University, and by the János Bolyai Scholarship of the Hungarian Academy of Sciences.	E-mail: {\tt zoltan.lorant.nagy@ttk.elte.hu}}
	\and Dávid R. Szabó\thanks{ELTE Linear Hypergraphs  Research Group,
		E\"otv\"os Lor\'and University, Budapest, Hungary \& HUN-REN Alfréd Rényi Institute of Mathematics 
		(Hungary, 1053 Budapest, Reáltanoda u. 13-15.) The author is supported by  the University Excellence Fund of Eötvös Loránd University. E-mail: {\tt szabo.r.david@gmail.com}}}
\date{}

\begin{document}
	
	\maketitle
	
	\begin{abstract}
		What is the maximum number of points that can be selected from an $n\times n$ square lattice so that no $k+1$ of them are collinear? 
		This was asked more than $100$ years ago for $k=2$ and it remained wide open ever since. In this paper, we prove  the exact answer is $kn$, provided that $ n\ge k\ge C\log n$ for an absolute constant $C<20$. The proof relies on carefully constructed block-uniform random bipartite graphs and concentration inequalities.  \\ MSC 2020: 52C10, 05C80.
	\end{abstract}
	%\subjclass[2020]{Primary 52C10; Secondary 05C80, 51E21.}
	%
	\section{Introduction}

	%et us
	Let us say that a set of points is in \emph{general position} if there are no three points in the set lying on a common line.    Dudeney \cite{dudeney1958amusements} asked  the following problem, formulated concerning the placement of chess pieces: what is the maximum number of points that can be placed in an $n \times n$ grid so that the points are in general position? This  problem was referred to as the No-Three-In-Line problem, and it has received considerable attention. We refer to the  book of Brass, Moser, and Pach \cite{Brass2005} and that of Eppstein  \cite{Eppstein:2018} on the history of this problem. 
	Brass, Moser and Pach called this one of the oldest and most extensively studied geometric questions concerning lattice points. For higher dimensional variants, we refer to
	\cite{por2007no}.
	
	It is straightforward to see that at most $2n$ points can be placed in general position by the pigeonhole principle.  On the other hand,  several examples have been constructed where the theoretical bound $2n$ can be attained, provided that $n$ is a small constant, see e.g. \cite{anderson1979update, flammenkamp1998progress}. However, it is still an open problem to determine the answer for general $n$.

	If $n$ is equal to a prime number,  Erdős observed that  the \emph{modular parabola},
	consisting of the points $(i,i^{2})\pmod p$, yields $n$ points in general position in $\zz^{2}\cap[1,n]^{2}$, see \cite{Roth:1951}.
	If $n$ is not a  prime, then taking $p$ to be the largest prime not exceeding
	$n$  and applying the modular parabola in a $p\times p$ subgrid yields $n-o(n)$ points in general position.
	The best general lower bound for the No-Three-In-Line problem relies on a construction  from 1975, due to Hall, Jackson, Sudbery, and Wild \cite{hall1975some}.  Their construction is built up from carefully chosen configurations in  subgrids, where one places points on a hyperbola $xy\equiv t \pmod p$, where $p$ is a prime slightly smaller than $n/2$, giving $\frac{3}{2}n - o(n)$ points in general position. 
	Interestingly, even what to conjecture on the \textit{asymptotic value} of the answer for the No-Three-In-Line problem is far from clear, with prominent experts believing it could be roughly $1.5n$ via the previous construction (see Green \cite{green100}), or it could be $2n$, thus the upper bound can be attained asymptotically in general \cite{Brass2005}, or somewhere in between \cite{Eppstein:2018, Guy/Kelly:1968}. Some related problems were studied recently in \cite{Aichholzer/Eppstein/Hainzl:2023} and \cite{nagy2023extensible}.

	We are interested in a generalization which might be the most natural one, where a collinear point set of $k+1$ points is forbidden in the grid $[1, n]\times [1, n]$.

	\begin{defi}
		Let $f_k(n)$ denote the maximum size  of a point set chosen from the vertices  of  the $n\times n$ grid $\cG=[1, n]\times [1, n]$ in which at most $k$ points lie on any line.
	\end{defi}
	
	This problem  has been studied before by Lefmann \cite{10.1007/978-3-540-68880-8_25, lefmann2012extensions}, in an even more general setting. He proved that for fixed $k$, $f_{k}(n)=\Omega(kn)$ holds.  
	Our result shows that the upper bound $kn$ is exact uniformly whenever $n\ge k$ and $k$ is at least a sufficiently large multiple of $\log n$.

	\begin{theorem}\label{main}
		Let $C^*:= ({\log(7/5)-2/7})^{-1}
		\approx 19.7$. 
		For every $C>  C^*$, 
		there exists a constant $N_C$ such that, whenever
		$n\ge N_C$ and $C\log n\le k\le n$, we have
		\[
		f_k(n)=kn.
		\]
	\end{theorem}

	To put this result into perspective, we recall the projective and affine variants of the problem, which are also of importance.
	For a given set of points, a line is called a \textit{secant} if it intersects the point set in at least two points.  Point sets having no secants of size larger than $2$ are called \textit{arcs}, while in general, point sets of a projective plane having no secants of size larger than $k$ are called $k$-\textit{arcs}. Following earlier work of Bose \cite{bose1947mathematical}, Segre showed that for odd $q$, every arc of $q+1$ points in a Desarguesian projective geometry $\PG(2,q)$
	is a conic \cite{Segre:1955a, Segre:1955}. For general $k$, the theoretical maximum size of a $k$-arc, also based on the pigeonhole principle, is $(q+1)(k-1)+1$ both in the projective and affine setting, however this can only be attained if $q$ is a power of $2$ and $k$ is a divisor of $q$ due to the results of Denniston, and of Ball, Blokhuis and Mazzocca \cite{ball1997maximal, ball2005bounds, denniston1969some}. 
	This in turn gives  the following consequence:
	
	\begin{theorem} Let $p$ be an odd prime number.
		In  the affine plane $\mathbb{F}_p^2$, the trivial upper bound $1+(p+1)(k-1)$ cannot be attained for the no-$(k+1)$-in-line problem, provided that $2\le k\le p-1$.
	\end{theorem}
	
	\paragraph{Outline}
	We finish the section by outlining the proof of our main theorem and describing the organisation of the paper.
	The first main idea is to consider the problem in a graph-theoretical setting, where we assign $n+n$ vertices to the $n+n$ axis-parallel lines of the grid $[1,n]\times[1,n]$, and take their intersection graph $K_{n,n}$. 
	
	\begin{problem}\label{graph_p}
		What is the maximum number of edges of a spanning subgraph $G$ of a balanced bipartite graph $K_{n,n}$, where each vertex has degree $\le k$ and for a given set $\mathcal{M}$ of  matchings of $K_{n,n}$, every matching intersects $E(G)$ in  at most $k$ edges?
	\end{problem}
	
	\noindent To obtain the no-$(k+1)$-in-line problem, we let $\mathcal{M}$ consist of the matchings corresponding to collinear grid points  $\ell \cap ([1,n]\times[1,n])$ for non-axis-parallel lines $\ell$.
	
	The second main idea is to choose the point set of size $kn$ from the grid (or the $k$-regular subgraph of $K_{n,n}$) in a random manner, in such a way that the expected number of the chosen points from $\ell$ is significantly less than $k$ for every line $\ell$ which is not horizontal or vertical. To do so, we design block-uniform random graphs (\autoref{sec:4|n}), in which the probability that an individual edge is selected depends on its block.
	After this step, we apply an elementary factorial-moment estimate for intersections of fixed matchings with uniformly random regular bipartite graphs. 
	
	In our approach, the existence of a suitable random point set demands a divisibility condition on the size of the grid. Thus we also show a method which provides suitable  sets of $kn'$ points in an $n'\times n'$ grid from previously constructed suitable sets  of size $kn$ in an $n\times n$ grid, where $n$ is only slightly smaller than $n'$ and satisfies the required divisibility condition.

	Our paper is organized as follows. In \autoref{sec:prelim}, we set the notation and introduce the main tools for the proof. We carry out the case when $k$ is very large, i.e., $k\ge \frac{2}{3}n$ in \autoref{sec:nagyk}. 
	\autoref{sec:4|n} is devoted to the proof of \autoref{main} in the case $k$ is not very large ($k< \frac{2}{3}n$; actually, we only need  $k< \frac{5}{6}n$) assuming certain divisibility conditions on $n$ and $k$. 
	In \autoref{sec:generalN}, we extend our result by dropping the divisibility conditions on $n$ and $k$.  Finally, we discuss some open problems and alternative proof techniques in \autoref{sec:openQUestions}.
	\section{Preliminaries}\label{sec:prelim}
	
	\subsection{Notation}
	If $a,b\in \zz$, then the notation $[a,b]$ will refer to the integer interval $\{x\in \zz: a\le x\le b\}$.
	A \textit{grid} is a set $\cG=I_1\times I_2$ where $I_1=[a_1,b_1]$ and $I_2=[a_2,b_2]$ for some integers $a_1\le b_1$ and $a_2\le b_2$. We say that $\cG$ is an $n_1\times n_2$ grid, where $n_j=|I_j|=b_j-a_j+1$ for $j=1,2$. We call $\cG$ a \textit{square grid} if $n_1=n_2$.
	
	We will often identify a subset $S$ of a grid $\cG$ with a bipartite graph with vertex classes $I_1$ and $I_2$, where edge $(i_1,i_2)$ appears in the graph if and only if $(i_1,i_2)\in S$. An \textit{$r$-factor} in a bipartite graph is a spanning subgraph where every vertex has degree $r$. Analogously, an $r$-factor in a grid $\cG$ is a subset of $\cG$ containing exactly $r$ points in each row and column. A \textit{$t$-matching} in a bipartite graph is a matching of size $t$, while in a grid it corresponds to $t$ points with no two lying on the same axis-parallel line.
	
	A \textit{secant} $\ell$ for a grid $\cG$ is a line of the Euclidean plane that meets $\cG$ in at least two points. We call a line $\ell$ \textit{generic} if it is not horizontal or vertical. Each generic secant $\ell$ has a unique direction vector $(u, v)$ such that $u\in \zz^{+}$, $v\in \zz\setminus \{0\}$, and the vector is \textit{primitive}, i.e., $\gcd(u,v)=1$. In this case, call $M=\max(|u|,|v|)$ the \textit{modulus} of $\ell$. Note that $\ell$ intersects an $n\times n$ square grid in at most $\left\lceil \frac{n}{M}\right\rceil$ points. Let $\cL(\cG)$ be the set of all generic secants of $\cG$.%, and for a fixed real $\kappa\ge 0$, let $\cL_{>\kappa}(\cG)=\{\ell\in \cL(\cG): |\ell\cap \cG|>\kappa\}$ be defined as the generic secants admitting more than $\kappa$ grid points.
	
	A set $S\subseteq \cG$ is a \emph{no-$(k+1)$-in-line set} if $|S\cap \ell|\leq k$ for every Euclidean line $\ell$. 
	A no-$(k+1)$-in-line set $S\subseteq \cG$ is said to have \emph{reserve $h$} for some $h\in \N$, if $|S\cap \ell|\leq k-h$ for every $\ell\in\cL(\cG)$.

	For events $E_n$ indexed with a parameter $n$, we say that $E_n$ occurs \textit{with high probability} (w.h.p.) if the probability of $E_n$ tends to $1$ as $n\to\infty$.
	
	Throughout the paper, $\log$ denotes the natural logarithm. Let us introduce the function $\Phi(x):=\log(1/x)-1+x$ for $x\in (0,1)$.
	
	\subsection{Tools}
	
	We shall need only two elementary ingredients: a factorial-moment estimate for random regular bipartite graphs and a count of grid secants by their direction.
	
	\subsubsection{A fixed matching in a random regular bipartite graph}
	
	For integers $a\ge0$ and $t\ge0$, write $(a)_t=a(a-1)\cdots(a-t+1)$, with $(a)_0=1$.
	
	\begin{lemma}[Fixed submatching bound]\label{lem:fixed-submatching}
		Let $G$ be chosen uniformly from the $r$-regular bipartite subgraphs of $K_{m,m}$, and let $T$ be a fixed matching of size $t$. Then
		\[
		\PP(T\subseteq E(G))\le \frac{r^t}{(m)_t}.
		\]
	\end{lemma}
	
	\begin{proof}
		By symmetry, every $t$-matching of $K_{m,m}$ has the same probability, say $p_t$, of being contained in $G$. The number of $t$-matchings of $K_{m,m}$ is
		\[
		\binom mt^2t!=\binom mt(m)_t.
		\]
		Every $r$-regular bipartite graph contains at most $\binom mt r^t$ such matchings: choose their left endpoints and then choose one incident edge at each selected endpoint, allowing overcounting when right endpoints collide. Hence
		\[
		\binom mt(m)_t p_t
		=\EE[\text{number of $t$-matchings in $G$}]
		\le \binom mt r^t,
		\]
		which proves the claim.
	\end{proof}
	
	\begin{lemma}[Factorial moments across independent blocks]\label{lem:factorial-moment}
		Let $G_1,\dots,G_b$ be independent uniformly random $r_j$-regular bipartite subgraphs of copies of $K_{m,m}$. In the $j$th copy let $M_j$ be a fixed matching of size $s_j$, and put
		\[
		X_j=|M_j\cap E(G_j)|,\qquad X=\sum_{j=1}^bX_j,
		\qquad \mu=\sum_{j=1}^b\frac{r_j}{m}s_j.
		\]
		Then for every integer $0\le t<m$,
		\begin{equation}\label{eq:factorial-moment}
			\EE\binom Xt
			\le
			\exp\!\left(\frac{t^2}{m-t}\right)\frac{\mu^t}{t!}.
		\end{equation}
		Consequently, for integers $x\ge t$,
		\begin{equation}\label{eq:factorial-tail}
			\PP(X\ge x)
			\le
			\exp\!\left(\frac{t^2}{m-t}\right)\frac{\mu^t}{(x)_t}.
		\end{equation}
	\end{lemma}
	
	\begin{proof}
		Expanding the binomial coefficient of a sum and using independence gives
		\[
		\EE\binom Xt
		=\sum_{t_1+\cdots+t_b=t}\prod_{j=1}^b\EE\binom{X_j}{t_j}.
		\]
		By \autoref{lem:fixed-submatching},
		\[
		\EE\binom{X_j}{t_j}
		\le \binom{s_j}{t_j}\frac{r_j^{t_j}}{(m)_{t_j}}.
		\]
		For $t_j\le t<m$,
		\[
		\frac{m^{t_j}}{(m)_{t_j}}
		=\prod_{q=0}^{t_j-1}\left(1-\frac qm\right)^{-1}
		=\prod_{q=0}^{t_j-1}\left(1+\frac{q}{q-m}\right)
		\le \prod_{q=0}^{t_j-1} \left(1+\frac{q}{m-t}\right)
		< \left(1+\frac{t_j}{m-t}\right)^{t_j}
		\le \exp\!\left(\frac{t_j^2}{m-t}\right).
		\]
		Since $\sum_jt_j^2\le t^2$ and $\binom{s_j}{t_j}\le s_j^{t_j}/t_j!$, we obtain
		\begin{align*}
			\EE\binom Xt
			&\le \exp\!\left(\frac{t^2}{m-t}\right)
			\sum_{t_1+\cdots+t_b=t}
			\prod_{j=1}^b\frac{(s_jr_j/m)^{t_j}}{t_j!}\\
			&=\exp\!\left(\frac{t^2}{m-t}\right)\frac{\mu^t}{t!}.
		\end{align*}
		Since
		\[
		\binom Xt \geq 
		\mathbf 1_{\{X\ge x\}}\binom Xt
		\ge
		\mathbf 1_{\{X\ge x\}}\binom xt,
		\]
		taking expectations gives
		\[
		\EE\binom Xt
		\geq
		\binom xt\,\PP(X\ge x).
		\]
		Substituting \eqref{eq:factorial-moment} yields
		\eqref{eq:factorial-tail}.
	\end{proof}
	
	Recall that we defined $\Phi(q)\coloneqq\log(1/q)-1+q$ for $0<q<1$.
	The next consequence is the form used in the logarithmic range of $k$.
	
	\begin{cor}[Poisson-type upper tail]\label{cor:poisson-tail}
		Fix $0<\delta<1$ and a nonnegative integer $h$. In the setting of
		\autoref{lem:factorial-moment}, suppose that $k\to\infty$, $k=o(m)$,
		$x=k-h$, and $\mu\le(\delta+o(1))k$%, where the $o(1)$ is uniform over the family under consideration
		. Then
		\[
		\PP(X\ge x)
		\le \exp\bigl(-(\Phi(\delta)-o(1))k\bigr).
		\]
	\end{cor}
	
	\begin{proof}
		Apply \eqref{eq:factorial-tail} with
		$t=\lfloor(1-\delta)x\rfloor$. We have $t^2/(m-t)<k^2/(m-k)=o(k)$. Stirling's formula gives
		\[
		\log(x)_t
		=x\log x-(x-t)\log(x-t)-t+O(\log x).
		\] Moreover, $\mu\le(\delta+o(1))x$, and hence
		$ \log\mu\le \log x+\log\delta+o(1).$
		As $x-t=\delta x+O(1)$, it follows that
		\[
		\log\frac{\mu^t}{(x)_t}
		\le
		-\bigl(\log(1/\delta)-1+\delta\bigr)x+o(k)
		=-\Phi(\delta)x+o(k).
		\]
		As $x=k-O(1)$, the result follows.
	\end{proof}
	
	For the range in which the degrees are not $o(m)$, the following coarser estimate is sufficient.
	
	\begin{prop}[Uniform additive deviation]\label{prop:uniform-additive}
		Let $M_0$ be a fixed matching of size $m_0$ in $K_{m,m}$, let $G$ be a uniformly random $r$-regular bipartite subgraph, and put $X=|M_0\cap E(G)|$. For every $0<\lambda\le1$,
		\[
		\PP\left(X>\left(\frac{m_0}{m}+\lambda\right)r\right)
		\le 2\exp\left(-\frac{\lambda^2r}{80}\right).
		\]
	\end{prop}
	
	\begin{proof}
		Put $y=\lceil(m_0/m+\lambda)r\rceil$ and
		$t=\lfloor\lambda r/8\rfloor$. If $y>m_0$, the event is empty.
		Otherwise, \autoref{lem:fixed-submatching} gives
		\[
		\PP(X\ge y)
		\le \frac{\binom{m_0}{t}}{\binom yt}\frac{r^t}{(m)_t}
		\le \left(\frac{m_0r}{(y-t+1)(m-t+1)}\right)^t.
		\]
		Assume first that $\lambda r\ge16$. The event being nonempty implies
		$(m_0/m+\lambda)r\le m_0$, so $r\le m$, and therefore
		$t/m\le\lambda/8$. Also $t/r\le\lambda/8$. Writing $a=m_0/m\leq 1$, the expression in parentheses is at most
		\[
		\frac{a}{(a+\lambda/2)(1-\lambda/8)}
		\le \frac{1}{(1+\lambda/2)(1-\lambda/8)}
		\le \left(1+\frac{5\lambda}{16}\right)^{-1}.
		\]
		Now $t\ge\lambda r/16$ and
		$\log(1+5\lambda/16)\ge\lambda/5$, so the bound on
		$\PP(X\ge y)$ above is at most
		$\exp(-\lambda^2r/80)$. If $\lambda r<16$, the stated right-hand side is greater than $1$, and the result is trivial.
	\end{proof}
	
	There are two further natural routes to estimates of this kind. The first is suggested by the sandwiching philosophy of Kim and Vu \cite{kim2004sandwiching}: if a random regular graph can be coupled between two binomial random graphs of essentially the same edge density, then upper-tail estimates for its intersection with a fixed sparse subgraph may be transferred from the Erdős-Rényi model. A corresponding theory for biregular bipartite graphs was developed by Klimošová, Reiher, Ruciński and Šileikis \cite{klimovsova2023sandwiching}; in the parameter ranges covered by such sandwiching results, this provides an alternative route to the concentration estimates required here. A second route is provided by asymptotic enumeration. Results of Greenhill and McKay \cite{greenhill2009random}, McKay \cite{mckay2010subgraphs}, and Isaev and McKay \cite{isaev2018complex} estimate the probability that a uniformly random graph with prescribed degrees contains a given sparse subgraph. In the particular case of a matching, this probability can be expressed as a ratio of enumeration formulae obtained by reducing the prescribed degree by one at each endpoint of the matching. These results motivate an estimate of order $(r/m)^{|T|}$ for the probability that a fixed matching $T$ is contained in a uniformly random $r$-regular bipartite graph on $m+m$ vertices. For the logarithmic regime of $k$ as a function of $n$ considered here, however, the direct double-counting argument in \autoref{lem:fixed-submatching} is more suitable: it gives the non-asymptotic bound
	\[
	\mathbb P\bigl(T\subseteq E(G)\bigr)
	\le \frac{r^{|T|}}{(m)_{|T|}}
	\]
	without an asymptotic error term or additional uniformity conditions.
	
	\subsubsection{Counting secants by modulus}
	
	\begin{lemma}\label{lem:modulus-count}
		The number of generic secants of $[1,n]^2$ whose modulus is at most $D$ is $O(nD^3)$, with an absolute implied constant.
	\end{lemma}
	
	\begin{proof}
		There are $O(M)$ primitive direction vectors of modulus exactly $M$. For a fixed direction $(u,v)$, every secant has a unique first grid point in that direction, and such a point lies within $|u|$ columns or $|v|$ rows of the boundary. Thus there are $O(nM)$ secants with that direction. Summing $O(nM^2)$ over $M\le D$ proves the claim.
	\end{proof}

	\section{Proof of \autoref{main}}
	
	We prove the theorem using two block-uniform random constructions (\autoref{sec:4|n}). The $5\times5$ construction controls the logarithmic range of $k$ and determines the constant $C_*$; the $4\times4$ construction handles the remaining nontrivial range. Both initially require divisibility assumptions. The resulting sets have a fixed reserve on generic secants, which lets us adjust $k$ and enlarge the grid in \autoref{sec:generalN}. The range $k\ge 2n/3$ is covered by the explicit construction of \autoref{sec:nagyk}.

	\subsection{Case \texorpdfstring{$k\ge \frac{2}{3}n$}{k>=2n/3}: explicit construction}\label{sec:nagyk}
	
	Here we solve the problem when $n-k\le n/3$, by showing a point set of size $(n-k)n$ which intersects every secant $\ell$ of the grid $\cG=[1,n]^2$ in at least $|\ell\cap \cG|-k$ points. Thus the complement will be a set without collinear $(k+1)$-sets.
	
	\begin{prop}\label{thm:kLarge}
		Let $n\in \N$ be arbitrary. 
		If $k\in\N$ satisfies $\frac{2}{3} n\leq k\leq n$, then $f_k(n)=kn$.
	\end{prop}
	
	\begin{proof}
		Observe that it is enough to consider the lines which have at least $\frac23n$ grid points. These include the horizontal and vertical lines, and the lines of slope $\pm 1$ that are close enough to the two main diagonals, i.e., the ones between the dashed lines in \autoref{fig:kLarge}. 
		\begin{figure}[ht]
			\centering

			\begin{tikzpicture}[scale=0.4]
				
				\def\myradius{1.3ex}
				
				\draw[opacity=0.5,gray] (1,1) grid (16,16);
				
				\foreach \x in {0,...,4} {
					\foreach \y in {0,...,4} {
						\fill[red] (1+\x,1+\y) circle (\myradius); 
						\fill[blue] (6+\x,7+\y) circle (\myradius); 
					}
				}
				
				\foreach \x in {0,...,4} {
					\foreach \y in {0,...,\x} {
						\fill[black] (15-\x,16-\y) circle (\myradius); 
						
						\ifnum \y>0
						\fill[black] (12+\x,11+\y) circle (\myradius);
						\else 
						\fill[black] (12+\x,6+\y) circle (\myradius);
						\fi        
					}
				}
				
				\node[red,anchor=north east]  at (1,1) {$\cP_0$};  
				\draw[red,<->] (1,0) -- (5,0) node[midway,anchor=north] {$n-k$};
				\draw[red,<->] (0,1) -- (0,5) node[midway,anchor=east] {$n-k$};
				\draw[red,dashed] (1,5) -- (12,16);
				\draw[red,dashed] (5,1) -- (16,12);
				
				\node[blue,anchor=south east] at (6,11) {$\cP_1$}; 
				\draw[blue,<->] (6,0) -- (10,0) node[midway,anchor=north] {$n-k$};
				\draw[blue,<->] (0,7) -- (0,11) node[midway,anchor=east] {$n-k$};
				\draw[blue,dashed] (1,12) -- (12,1);
				\draw[blue,dashed] (5,16) -- (16,5);
				
			\end{tikzpicture}
			
			\caption{The complement of a suitable construction 
				for $n=16, k=11$. }
			\label{fig:kLarge}
		\end{figure}
		Let us take the point set $\cP_2=\cP_0\cup \cP_1$, where $\cP_0$ and $\cP_1$ are two $(n-k)\times(n-k)$ square grids defined by 
		\begin{align*}
			\cP_0 &=[1,n-k]\times [1,n-k], \\
			\cP_1 &= [n-k+1,2n-2k]\times [2k-n+1,k]
		\end{align*}
		as indicated in \autoref{fig:kLarge}.

		It is easy to see that $\cP_0$ intersects every line $\ell[y=x\pm c]$ exactly $n-k-c$ times for $0\leq c\leq n-k$,
		similarly $\cP_1$ intersects every line $\ell[y=-x+n+1\pm c]$ exactly $n-k-c$ times for $0\leq c\leq n-k$.
		Moreover, every horizontal or vertical line contains either $0$ or $n-k$ points of $\cP_2$. In order to obtain a point set $\cP_3\supseteq \cP_2$ where every vertical and horizontal line contains exactly $n-k$ points, we add a suitable point set to $\cP_2$, which corresponds to an $(n-k)$-regular bipartite graph on $(n-2(n-k)) +(n-2(n-k))  $ vertices. Such a graph exists due to Kőnig's line coloring theorem.
		The complement $\cG\setminus \cP_3$ will thus be a no-$(k+1)$-in-line set.
	\end{proof}

	\subsection{Random block-uniform constructions}\label{sec:4|n}
	
	We now construct the required point sets under suitable divisibility assumptions. Two fixed block patterns will be presented. A $5\times5$ pattern gives the constant stated in the main theorem in the logarithmic range for $k$, while the  $4\times4$ pattern handles the remaining nontrivial range by a coarser estimate.
	
	\subsubsection{The general block-uniform model}
	
	Let $q\ge2$ divide $n$, and decompose $\cG=[1,n]^2$ into $q^2$ equal subgrids
	\[
	\cG_{i,j}=\left[(i-1)\frac nq+1,i\frac nq\right]
	\times
	\left[(j-1)\frac nq+1,j\frac nq\right],
	\qquad 1\le i,j\le q.
	\]
	Let $R=(r_{i,j})$ be a matrix of integers in $[0,n/q]$ whose row and column sums are all equal to $k$. Independently in each block $\cG_{i,j}$, choose a uniformly random $r_{i,j}$-factor $S_{i,j}$, and put
	\[
	S=\bigcup_{i,j=1}^qS_{i,j}.
	\]
	Then every row and column of $\cG$ contains exactly $k$ points of $S$. For each $i,j$, let $\alpha_{i,j}=\frac{r_{i,j}}{n/q}$.
	
	For a generic secant $\ell$, put
	\[
	g_{i,j}(\ell)=|\ell\cap\cG_{i,j}|,
	\qquad
	E_R(\ell)=\sum_{i,j}\alpha_{i,j}g_{i,j}(\ell).
	\]
	Thus $E_R(\ell)=\EE(|S\cap\ell|)$. 
	As the line $\ell$ is generic, no two points of $\ell\cap \cG_{i,j}$ lie in
	the same row or in the same column of $\cG_{i,j}$; hence
	$\ell\cap \cG_{i,j}$ is a matching in the copy of $K_{n/q,n/q}$
	corresponding to $\cG_{i,j}$.
	Since a line crosses at most $2q-1$ blocks, the random variables
	$|S_{i,j}\cap\ell|$ form a sum of at most $2q-1$ independent matching-intersection variables, to which \autoref{lem:factorial-moment} applies with $m=n/q$.

	\subsubsection{The \texorpdfstring{$4\times4$}{4 by 4} block construction}
	We present this conceptually simpler construction to highlight the essence of the proof. The optimized version will be carried out in the next subsection.
	
	Assume that $4\mid n$ and $10\mid k$. Decompose the grid into $4\times4$ blocks as in \autoref{fig:subgrids1}. Call a block sparse if it meets one of the two main diagonals and dense otherwise, and put
	\begin{equation}\label{eq:defR}
		r_{i,j}=\begin{cases}
			k/5,&\text{in a sparse block},\\
			3k/10,&\text{in a dense block},
		\end{cases}
		\qquad
		\text{i.e. let }
		\qquad 
		R_4=(r_{i,j})=\frac{k}{10}\begin{pmatrix}
			2&3&3&2\\
			3&2&2&3\\
			3&2&2&3\\
			2&3&3&2
		\end{pmatrix}.
	\end{equation}
	
	\begin{figure}[htb!]
		\centering
		\begin{tikzpicture}[scale=0.4]

			\draw[opacity=0.5,gray]  (0,0) -- (21,21) node[anchor=west,black,opacity=1] {$\ell_0$};
			\draw[opacity=0.5,gray]  (5,0) -- (21,16) node[anchor=west,black,opacity=1] {$\ell_{n/4}$};
			\draw[opacity=0.5,gray] (10,0) -- (21,11) node[anchor=west,black,opacity=1] {$\ell_{2n/4}$};
			\draw[opacity=0.5,gray] (15,0) -- (21, 6) node[anchor=west,black,opacity=1] {$\ell_{3n/4}$};
			\draw[opacity=0.5,gray] (18,0) -- (21, 3) node[anchor=west,black,opacity=1] {$\ell_{n-2}$};
			
			\draw[opacity=0.5,gray] (1,1) grid (20,20);
			\foreach \x in {1,...,4} {
				\foreach \y in {1,...,4} {
					\ifthenelse{\x=\y \OR \numexpr\x+\y\relax = 5}
					{\draw[draw=blue, very thick,rounded corners, dashed] 
						(-4.2+5*\x,-4.2+5*\y) rectangle ++(4.4,4.4) 
						node[pos=0.5, align=center] {\strut{}$\cG_{\x,\y}$\\\strut{}sparse};}
					{\draw[draw=red, fill=red, fill opacity=0.2,text opacity=1, very thick,rounded corners] 
						(-4.2+5*\x,-4.2+5*\y) rectangle ++(4.4,4.4)
						node[pos=0.5, align=center] {\strut{}$\cG_{\x,\y}$\\\strut{}dense};} 
				}
			}

		\end{tikzpicture}
		
		\caption{The sparse and dense blocks in the $4\times4$ construction.}
		\label{fig:subgrids1}
	\end{figure}
	
	\begin{lemma}\label{lem:R4-expectation}
		The matrix $R_4$ in \eqref{eq:defR} has every row and column sum equal to $k$, and for every generic secant $\ell$,
		\[
		E_{R_4}(\ell)\le\frac45k.
		\]
	\end{lemma}
	
	\begin{proof}
		Every block row and block column contains two sparse and two dense blocks, so the row and column sums are $k$. The construction is valid for $k\le5n/6$.
		
		If $\ell$ has modulus one, symmetry reduces to the family
		$y=x-c$. As $c$ varies, $E_{R_4}(\ell)$ is linear between consecutive block-boundary values $0,n/4,n/2,3n/4,n-2$ (cf. \autoref{fig:subgrids1}), and direct substitution at these values gives a maximum of $4k/5$. If the modulus is at least two, then $|\ell\cap\cG|\le n/2$ and
		\[
		E_{R_4}(\ell)
		\le \frac{6k}{5n}|\ell\cap\cG|
		\le\frac35k.
		\qedhere
		\]
	\end{proof}
	
	\begin{prop}[The $4\times 4$ block construction]\label{prop:medium-divisible}
		Suppose that $4\mid n$ and $10\mid k$, that
		\[
		\frac{k}{\log n}\longrightarrow\infty,
		\]
		and that $   k\le \frac56 n.$
		Then the $4\times4$ random construction produces, with high probability,
		a no-$(k+1)$-in-line set of size $kn$ having reserve $15$.
	\end{prop}
	
	\begin{proof}
		Fix a generic secant $\ell$. It meets at most seven blocks. In every intersected block let
		$X_{i,j}=|S_{i,j}\cap\ell|$ and
		$\mu_{i,j}=\EE X_{i,j}$. Put $d=k/70$. By \autoref{prop:uniform-additive},
		\[
		\PP(X_{i,j}>\mu_{i,j}+d)
		\le 2\exp\left(-\frac{d^2}{80r_{i,j}}\right)
		\le 2\exp\left(-\frac{k}{117600}\right),
		\]
		where we used $r_{i,j}\le3k/10$. If none of these block events occurs, then by \autoref{lem:R4-expectation},
		\[
		|S\cap\ell|
		\le \frac45k+7\frac{k}{70}
		=\frac9{10}k
		\le k-15
		\]
		for all sufficiently large $n$. Hence the probability that some generic secant violates the reserve is at
		most
		\[
		14n^4\exp\left(-\frac{k}{117600}\right).
		\]
		This tends to zero whenever $k/\log n\to\infty$.
	\end{proof}
	
	In the statement, we choose the value $15$
	because after rounding $k$ upward by at most $9$ in the proof of
	\autoref{main}, a reserve of at least $6$ remains. This is exactly what
	is needed to enlarge the grid by at most $3$ rows and $3$ columns in
	\autoref{lem:adjustN}.
	
	\subsubsection{The \texorpdfstring{$5\times5$}{5 by 5} block construction}
	Now we optimize the block-uniform construction in order to obtain the best constant factor in the lower bound with this technique.
	Assume that $5\mid n$ and $14\mid k$, and define
	\begin{equation}\label{eq:R5}
		R_5=\frac{k}{14}
		\begin{pmatrix}
			0&5&4&5&0\\
			5&2&0&2&5\\
			4&0&6&0&4\\
			5&2&0&2&5\\
			0&5&4&5&0
		\end{pmatrix}.
	\end{equation}
	All row and column sums are $k$. The largest entry is $3k/7$, so the construction is defined whenever $k\le7n/15$.
	
	\begin{lemma}[Expectation bound for $R_5$]\label{lem:R5-expectation}
		For every generic secant $\ell$,
		\[
		E_{R_5}(\ell)\le \frac57k+\frac{135}{7}\frac{k}{n}.
		\]
	\end{lemma}
	
	\begin{proof}
		Let $B$ denote the integer matrix in \eqref{eq:R5}. Let the primitive direction vector of $\ell$ be $(u,v)$, and its modulus $M:=\max(|u|,|v|)$. We have
		\[
		E_{R_5}(\ell)=\frac{5k}{14n}\sum_{i,j}B_{i,j}g_{i,j}(\ell).
		\]
		%Scale the grid by the factor $5/n$, so that the block decomposition becomes the unit-square decomposition of $[0,5]^2$. Parametrise a line with unit speed in the $\ell_\infty$ norm. If $W(u,v)$ denotes the maximum weighted length of such a chord, with cell weights given by $B$, then each block count differs from its scaled chord length by at most two. Since a line crosses at most nine blocks and $\max B_{i,j}=6$,
		
		Apply the scaling map \(T(z)=\frac5n z\), and write
		\(\widetilde\ell=T(\ell)\).  For \(1\le i,j\le5\), let $C_{i,j}=[i-1,i]\times[j-1,j]$, forming a decomposition of $[0,5]^2$ into unit squares. Note that \(T(\cG_{i,j})\subseteq C_{i,j}\).
		
		Choose \(h\in\{1,2\}\) so that the \(h\)th coordinate of \((u,v)\)
		has absolute value \(M\), and let \(\pi_h\) denote projection onto this
		coordinate.  For a line \(L\) parallel to \((u,v)\), let $L_{i,j}(L)$ be the length of $\pi_h(L\cap C_{i,j})$.
		
		The images under \(T\) of consecutive lattice points on \(\ell\)
		have projected coordinates separated by exactly \(5M/n\).  Moreover,
		every point counted by \(g_{i,j}(\ell)\) is mapped into \(C_{i,j}\).
		Since \(\pi_h(\widetilde\ell\cap C_{i,j})\) is an interval, it follows
		that
		\[
		g_{i,j}(\ell)
		\le \frac{n}{5M}L_{i,j}(\widetilde\ell)+1.
		\]
		
		Define
		\[
		W(u,v):=\sup_{L\parallel(u,v)}
		\sum_{i,j}B_{i,j}L_{i,j}(L).
		\]
		Let $\mathcal A(\ell)=\{(i,j):g_{i,j}(\ell)>0\}$. It is then easy to see that $|\mathcal A(\ell)|\le 2\cdot 5-1=9$. Hence
		\begin{equation}\label{eq:discrete-chord}
			\begin{aligned}
				\sum_{i,j}B_{i,j}g_{i,j}(\ell)
				&\le \frac{n}{5M}
				\sum_{(i,j)\in\mathcal A(\ell)}
				B_{i,j}L_{i,j}(\widetilde\ell)
				+\sum_{(i,j)\in\mathcal A(\ell)}B_{i,j} \\
				&\le \frac{n}{5M}W(u,v)+9\max_{i,j}B_{i,j} \\
				&=\frac{n}{5M}W(u,v)+54.
			\end{aligned}
		\end{equation}
		
		\noindent For \(M\le2\), let $\ell_c$ denote the line $vx-uy=c$.
		The function $c\mapsto \sum_{i,j} B_{i,j} L_{i,j}(L)$ is affine between consecutive values of the form $c=va-ub$ with $a,b\in\{0,\ldots,5\}$, corresponding to lines through the vertices of the unit-square decomposition. Direct evaluation at these values gives
		\[W(u,v)=\begin{cases}
			10,\text{ if } (u,v)=(1,\pm 1),\\
			18,\text{ if } (u,v)\in\{(1,\pm 2), (2,\pm 1)\}.\\
		\end{cases}\]
		%\[
		%\begin{array}{c|c}
		%\text{primitive direction}&W(u,v)\\ \hline
		%(1,\pm 1)&10\\
		%(1,\pm 2)&18\\
		%(2,\pm 1)&18
		%\end{array}
		%\]
		For \(M\ge3\), the total projected length of a secant of \([0,5]^2\)
		is at most \(5\), and therefore
		\[
		W(u,v)\le5\max_{i,j}B_{i,j}=30\le10M.
		\]
		Thus \(W(u,v)\le10M\) in every case.  Substitution into
		\eqref{eq:discrete-chord} yields
		\[
		\sum_{i,j}B_{i,j}g_{i,j}(\ell)\le2n+54,
		\]
		and consequently
		\[
		E_{R_5}(\ell)
		\le\frac57k+\frac{135}{7}\frac{k}{n}.
		\qedhere
		\]
	\end{proof}
	
	\begin{prop}[Low range with divisibility]\label{prop:low-divisible}
		Let
		\[
		C_*\coloneqq\frac{1}{\Phi(5/7)}
		=\frac{1}{\log(7/5)-2/7}.
		\]
		For every $C>C_*$ there is $n_C$ such that, whenever
		\[
		n\ge n_C,\qquad 5\mid n,\qquad 14\mid k,
		\qquad C\log n\le k\le\frac{n}{\log n},
		\]
		the random construction based on $R_5$ produces, with high probability, a no-$(k+1)$-in-line set of size $kn$ having reserve $21$.
	\end{prop}
	
	\begin{proof}
		Put $x=k-20$ and $D=\lceil\log n\rceil$. First consider secants of modulus at most $D$. Their number is $O(nD^3)$ by \autoref{lem:modulus-count}. By \autoref{lem:R5-expectation}, their expected intersection with $S$ is at most
		\[
		\left(\frac57+o(1)\right)k.
		\]
		Here the number of blocks is fixed, $m=n/5$, and $k=o(m)$. Thus \autoref{cor:poisson-tail} gives, uniformly over these secants,
		\[
		\PP(|S\cap\ell|\ge k-20)
		\le \exp\left(-\bigl(\Phi(5/7)-o(1)\bigr)k\right).
		\] 
		The union bound over the $O(n(\log n)^3)$ such lines tends to zero because
		$C\Phi(5/7)>1$.
		
		Now suppose that the modulus of $\ell$ is greater than $D$. Since
		$\max\alpha_{i,j}=15k/(7n)$ and
		$|\ell\cap\cG|\le\lceil n/D\rceil$, its mean satisfies
		\[
		\mu_\ell\le \frac{15k}{7n}\left(\frac nD+1\right)
		\le \frac{3k}{D}
		\]
		for large $n$. Apply \autoref{lem:factorial-moment}, specifically
		\eqref{eq:factorial-tail} with
		$t=\lfloor x/2\rfloor$. As $(x)_t\ge(x/2)^t$ and
		$t^2/(m-t)=o(k)$,
		\[
		\PP(|S\cap\ell|\ge x)
		\le \exp(o(k))\left(\frac{7}{D}\right)^{t}
		=\exp\bigl(-\Omega(k\log\log n)\bigr).
		\]
		There are fewer than $n^4/2$ secants altogether, so the union bound over the remaining lines also tends to zero. Horizontal and vertical lines contain exactly $k$ selected points, completing the proof.
	\end{proof}
	
	Similarly to the  \autoref{prop:medium-divisible}, we choose the reserve value $21$ because rounding $k$
	upward by at most $13$ in the proof of \autoref{main} leaves reserve at
	least $8$, sufficient to enlarge the grid by at most $4$ rows and
	$4$ columns in \autoref{lem:adjustN}.

	\subsection{Removing the divisibility assumptions}\label{sec:generalN}
	
	We use the following two elementary adjustment lemmas.
	
	\begin{lemma}[Adjusting $k$]\label{lem:adjustK}
		Let $0\le h\le k\le n$ be integers. Assume that
		$S\subseteq[1,n]^2$ is a no-$(k+1)$-in-line set of size $kn$ having reserve $h$. Then, whenever whenever integers $k', h'$ satisfy
		$k-k'=h-h'\ge0$ with  $0\le h'\le k'$, there is a no-$(k'+1)$-in-line set
		$S'\subseteq[1,n]^2$ of size $k'n$ having reserve $h'$.
	\end{lemma}
	
	\begin{proof}
		By Kőnig's line-colouring theorem, the $k$-regular bipartite graph corresponding to $S$ contains a $(k-k')$-factor $F$. Put $S'=S\setminus F$. Then $S'$ is a $k'$-factor, and every generic secant contains at most
		\[
		k-h=k'-h'\le k'
		\]
		points of $S'$. Thus $S'$ has the required size and reserve.
	\end{proof}
	
	\begin{lemma}[Adjusting $n$]\label{lem:adjustN}
		Let $0\le h''< k'\le n$ be integers with $h''$ even. If
		$S'\subseteq[1,n]^2$ is a no-$(k'+1)$-in-line set of size $k'n$ having reserve $h''$, then there is a no-$(k'+1)$-in-line set
		$S''\subseteq[1,n+h''/2]^2$ of size
		$k'(n+h''/2)$.
	\end{lemma}
	
	\begin{proof}
		Decompose the $k'$-factor $S'$ into disjoint $1$-factors
		$\cP_1,\dots,\cP_{k'}$. For each $1\le i\le h''/2$, choose
		$\cQ_i\subseteq\cP_i$ of size $k'$. Delete every point $(x,y)\in\cQ_i$ and add its two orthogonal projections $(n+i,y)$ and $(x,n+i)$. Let $S''$ be the resulting set in
		$[1,n+h''/2]^2$.
		
		Every old and new row and column contains exactly $k'$ points. A generic line contains at most $k'-h''$ surviving points in the old square, at most $h''/2$ points in the new columns, and at most $h''/2$ points in the new rows. This also holds if the generic line intersects the old square in only a single point as $k'-h''\ge 1$ by the assumption. Hence it contains at most $k'$ points of $S''$.
	\end{proof}
	
	\begin{proof}[Proof of \autoref{main}]

		Recall that
		\[
		C_*=\frac{1}{\log(7/5)-2/7}.
		\]
		We prove that for every $C>C_*$ there exists $N_C$ such that, whenever
		\[
		n'\ge N_C
		\qquad\text{and}\qquad
		C\log n'\le k'\le n',
		\]
		one has $f_{k'}(n')=k'n'$.
		
		Fix $C>C_*$, let $n'$ be sufficiently large, and let $k'$ satisfy
		\[
		C\log n'\le k'\le n'.
		\]
		The upper bound $f_{k'}(n')\le k'n'$ follows from the rows, so it remains
		to construct a set of that size.
		
		If $k'\ge2n'/3$, apply \autoref{thm:kLarge}. We therefore assume
		$k'<2n'/3$.
		
		\smallskip
		\noindent\emph{Low range.}
		Suppose first that
		\[
		k'\le\frac{n'}{2\log n'}.
		\]
		Put
		\[
		n=5\left\lfloor\frac{n'}5\right\rfloor,
		\qquad
		k=14\left\lceil\frac{k'}{14}\right\rceil.
		\]
		Then $n'-4\le n\le n'$, $k'\le k\le k'+13$, and, for large $n'$,
		\[
		C\log n\le k\le\frac{n}{\log n}.
		\]
		By \autoref{prop:low-divisible}, there is a no-$(k+1)$-in-line set
		$S\subseteq[1,n]^2$ of size $kn$ having reserve $21$. Applying
		\autoref{lem:adjustK} gives a no-$(k'+1)$-in-line set in $[1,n]^2$ with reserve at least
		$21-13=8$. Finally,
		\[
		h''=2(n'-n)\le8,
		\]
		so \autoref{lem:adjustN} enlarges the grid to $[1,n']^2$ and gives a set of size $k'n'$.
		
		\smallskip
		\noindent\emph{Medium range.}
		Now suppose that
		\[
		\frac{n'}{2\log n'}<k'<\frac23n'.
		\]
		Put
		\[
		n=4\left\lfloor\frac{n'}4\right\rfloor,
		\qquad
		k=10\left\lceil\frac{k'}{10}\right\rceil.
		\]
		Then $n'-3\le n\le n'$ and $k'\le k\le k'+9$. For all sufficiently large $n'$,
		\[
		\frac{n}{3\log n}\le k\le\frac56n.
		\]
		By \autoref{prop:medium-divisible}, there is a no-$(k+1)$-in-line set in $[1,n]^2$ of size $kn$ having reserve $15$. After applying \autoref{lem:adjustK}, the reserve is at least $15-9=6$, while
		$2(n'-n)\le6$. A final application of \autoref{lem:adjustN} completes the construction.
	\end{proof}
	
	\begin{remark} \label{rem:constant}
		The logarithmic threshold is determined by the $5\times5$ pattern alone. Its maximal asymptotic expected density on a generic line is $\delta=5/7$ (\autoref{lem:R5-expectation}), and the factorial-moment estimate gives the Poisson upper-tail exponent $ \Phi(\delta)=\log(1/\delta)-1+\delta.$
		Only $O(n(\log n)^3)$ lines have modulus at most $\log n$; all larger-modulus lines have a much smaller mean and contribute negligibly. Consequently the union bound requires precisely
		$C\Phi(5/7)>1$, giving
		$ C>C_*\approx 19.7$.
		No optimisation of the coarse medium-range estimate affects this constant.
	\end{remark}

	\section{Concluding remarks}\label{sec:openQUestions}

	The constant $C_*\approx19.7$ is the best constant obtained here from the fixed $5\times5$ pattern $R_5$ and the Poisson-type factorial-moment exponent. A smaller feasible block pattern density would improve it. More precisely, a fixed pattern with maximal expected line density $\delta<1$ would lead, by the same proof, to the threshold $ C>({\log(1/\delta)-1+\delta})^{-1}$.
	Determining the smallest feasible $\delta$ for general block patterns is a finite-dimensional optimisation problem for every fixed block size, but a complete analytic optimisation over all sizes is outside the scope of this paper.  
	
	\paragraph{Subsequent developments.}
	Grebennikov and Kwan proved that $f_k(n)=kn$ whenever
	$n\ge k\ge10^{37}$ \cite{grebennikov2025no}. %More recently, Ghosal, Goenka, Grebennikov, Keevash, Kwan and Pham resolved the problem for every fixed $k\ge3$ and all sufficiently large $n$ \cite{ghosal2026no}.
	The present argument gives a short, explicit and uniform proof in the growing-parameter regime $k\ge C\log n$.

\end{document}